\documentclass[graybox]{svmult}

\usepackage{type1cm}        
\usepackage{makeidx}         
\usepackage{graphicx}        
\usepackage{multicol}        
\usepackage[bottom]{footmisc}

\usepackage{newtxtext}       %
\usepackage[varvw]{newtxmath}       

\renewcommand{\D}{\mathbb{D}}

\makeindex             

\begin{document}

\title*{The Hilbert matrix done right}
\author{A. Montes-Rodr\'iguez and\\ J. A. Virtanen}
\institute{A. Montes-Rodr\'iguez \at University of Seville, Seville, Spain, \email{amontes@us.es}
\and J. A. Virtanen \at University of Reading, Reading, United Kingdom \email{j.a.virtanen@reading.ac.uk}}

%
%

\maketitle

\abstract*{We give very simple proofs of the classical results of Magnus and Hill on the spectral properties of the Hilbert matrix
$$
H = \left ( {1 \over i+j+ 1 } \right )_{i,j\geq 0}
$$
which defines a bounded linear operator on the sequence space $\ell^2$. In particular, we use the Mehler-Fock transform to find the spectrum and the latent eigenfunctions of the Hilbert matrix, that is, we show that the spectrum of $H$ is $[0,\pi]$ with no eigenvalues (Magnus' result) and describe all complex sequences $x$ such that $Hx=\mu x$ for some complex number $\mu$ (Hill's result).
}

\abstract{We give very simple proofs of the classical results of Magnus and Hill on the spectral properties of the Hilbert matrix
$$
H = \left ( {1 \over i+j+ 1 } \right )_{i,j\geq 0}
$$
which defines a bounded linear operator on the sequence space $\ell^2$. In particular, we use the Mehler-Fock transform to find the spectrum and the latent eigenfunctions of the Hilbert matrix, that is, we show that the spectrum of $H$ is $[0,\pi]$ with no eigenvalues (Magnus' result) and describe all complex sequences $x$ such that $Hx=\mu x$ for some complex number $\mu$ (Hill's result).
}


\section{Introduction}
\label{sec:1}
The classical Hilbert matrix 
$$
	H = \left ( {1 \over i+j+ 1 } \right )_{i,j\geq 0}
$$
was introduced in 1894 by Hilbert in connection with orthogonal polynomials studied by P\'olya. Magnus \cite{MR0041358}  found that the spectrum of the Hilbert matrix is the interval $ [0,\pi]$.   Previously, Hardy had observed that $H$ has no point spectrum. Subsequently, in [4] Hill solved the eigenvalue problem for the Hilbert matrix. In this note, using the Mehler-Fock transform, we will provide very simple proofs of both of these classical theorems. In particular, our approach shows that there is a close connection between the Mehler-Fock transform and the Hilbert matrix. Section 2 deals with Hill's Theorem on latent eigenfunctions while  Section 3 deals with Magnus' Theorem on the spectrum of the Hilbert matrix.

\section{The eigenfunctions}
\label{sec:2}

\subsection{The Hardy space $\mathcal{H}^2(\D)$}

In many situations, it is natural to translate the problem at hand to a space of functions in which many other tools become more readily available. Let $\D$ denote the open unit disk of the complex plane. Recall that the Hardy space $\mathcal H^2(\D)$ is the space of functions 
$$
	f(z)=\sum_{n=0}^\infty a_n z^n 
$$ 
analytic on $\D$ for which the norm
for which the norm 
$$
	\Vert f \Vert_{\mathcal H^2} ^2= \sum_{n=0}^\infty \vert a_n \vert ^2  
$$
is finite. It is obvious that $\mathcal H^2(\mathbb D)$ is isometrically isomorphic to $\ell^2$. In addition, it is well known that the functions in the Hardy space have boundary values almost everywhere (see, e.g., \cite{MR0924157}), and, in fact, the norm above has the integral representation
$$
\Vert f \Vert_{\mathcal H^2}^2= {1\over 2 \pi } \int_{-\pi}^\pi  \vert f(e^{i\theta } )\vert^2 \, d \theta.
$$
We recall that given two functions $f(z)= \sum_{n=0}^\infty a_n z ^n $ and $g(z)= \sum_{n=0}^\infty b_n z^n $, their inner product is given by 
$$
\langle f , g \rangle_{\mathcal H^2}= \sum_{n=0}^\infty a_n \bar   b_n= {1\over 2 \pi } \int_{-\pi }^\pi f(e^{i\theta}) \overline { g(e^{i\theta})} \, d \theta.
$$
We keep the same notation for the Hilbert matrix  $H$ when it is acting on $\mathcal H^2(\mathbb D)$, that is,
$$
 H  f (z)=\sum_{n=0}^\infty \left (  \sum_{m=0}^\infty {a_m \over n+m+1 } \right )z^n.
$$ 

\subsection{The latent eigenvalues of the Hilbert matrix}

We  will now state  Hill's theorem for the Hilbert matrix, see \cite[Thm.~1]{MR0121651}. 

\begin{theorem}\label{Hill}
For $0<\Re \mu \leq 1/2$,  the sequence $\{x_n\}_{n\geq 0}$ defined by 
$$
x_n=x_n( \mu)= \sum_{k=0}^n \binom {n}{k} (-1)^k {\Gamma(k+\mu) \Gamma(k+1-\mu) \over \Gamma(k+1)^2}
\quad \text { for } n\geq 0$$
satisfies
\begin{equation}\label{2.1}
	H \{x_n\}=\left \{{ \pi\over \sin (\pi \mu ) } x_n \right  \} 
\end{equation}
and $\{x_n\}\notin \ell^2$
\end{theorem}

\begin{remark}
We observe that for each complex number $M$ with $\Re M>0$, the equation $M=\pi / \sin(\pi \mu)$ has exactly one solution $\mu$  satisfying  $0<\Re \mu \leq 1/2$,  except for $M \in (0,\pi)$,  which have two solutions, but in such a case, the corresponding eigenfunctions for $\mu=1/2\pm i t$, $t>0$,   coincide.  
\end{remark}

\subsection{The eigenfunctions as Legendre functions}

Let us look for a compact form of the eigenfunctions. To this end, consider the function
\begin{equation}\label{2.2}
f_\mu  (z)=\sum_{n=0}^\infty  \sum_{k=0}^n \binom {n}{k} (-1)^k {\Gamma(k+\mu) \Gamma(k+1-\mu) \over \Gamma(k+1)^2 }z^n.
\end{equation}
Changing the order  of summation, we obtain
\begin{align*}
	f_\mu (z) &= \sum_{k=0}^\infty  \sum_{n=k}^\infty \binom {n}{k} (-1)^k {\Gamma(k+\mu) \Gamma(k+1-\mu) \over \Gamma(k+1) ^2}z^n \\
	&={1 \over 1-z}\sum_{k=0}^\infty { \Gamma(k+\mu) \Gamma(k+1-\mu) \over \Gamma(k+1)^2} \left (-z\over 1-z\right) ^k\\ 
	&= {\pi\over\sin(\pi \mu)}{1\over 1-z}\, _2F_1 \left (\mu,1-\mu, 1;  {-z \over 1-z}\right ),
\end{align*}
where 
$$
_2F_1(a,b;c,z)=\sum_{n=0}^\infty { \Gamma(a+n)\Gamma (b+n) \Gamma(c) \over \Gamma(a)\Gamma(b)\Gamma(\lambda +n)n ! }z^n
$$ 
is the hypergeometric function of parameters $a,b,c$. 
Now, applying Kummer's first formula, see \cite[Thm.~20, p.~60]{MR0107725}, for instance, we see that
\begin{equation}\label{2.3}
	f_ \mu(z)={\pi \over  \sin(\pi \mu )}(1-z)^{\mu-1}\, _2F_1(\mu,\mu ; 1 ;z).
\end{equation}
We point out here that although Kummer's formula is in principle true for $\vert z \vert <1$ and $\Re z<1/2$, since the series in \eqref{2.2} is analytic on a neighborhood of $0$ and the function in \eqref{2.3} is analytic on $\D$, we find that the equality holds on the whole unit disk $\D$. 

Since the eigenvalues are defined up to a non-zero scalar multiple, we may normalize redefining  
\begin{equation}\label{2.4}
	f_ \mu  (z)= (1-z)^{\mu-1}\, _2F_1(\mu,\mu ; 1 ;z).
\end{equation}
Looking at the coefficients of the eigenfunctions $f_\mu$, one easily verifies that the above eigenfunctions are not in $\mathcal H^2(\mathbb D)$, see \cite[Thm.~1]{MR0121651} and also \cite{MR2996436} for a different proof. Therefore, the eigenfunctions are latent. 
There are several ways to prove that the eigenfunctions have the form given in \eqref{2.4}; for instance, in \cite{MR2996436} instead of manipulating sums, the theory of ordinary differential equations was used.

Finally, using the definition of the Legendre functions, see \cite[pp.~165--167]{MR0350075}, and the fact that $P _ {-\mu}= P_{\mu-1}$, we have  
\begin{equation}\label{2.5}
	f_{\mu }(z)={1 \over 1- z } P_{\mu -1} \left ( 1+z \over 1-z \right )
\end{equation}
for $0<\Re \mu \leq 1/2$. 

\subsection{The Mehler-Fock transform}

Given a measurable function  $f:  [1,\infty) \to \mathbb C$, its Mehler-Fock transform is defined by 
$$
\hat f (t)=\mathcal P  f(t) = \int_1^\infty f(x) P_{it-1/2} (x)   \, dx ,  \quad {t \geq 0}, 
$$
whenever the integral exists. Here $P_{\nu }  $  is the  Legendre function of the first kind of order $\nu$.   A sufficient condition for the integral above to exist is that  $f(x)/\sqrt x$ belongs to $L^1[1,\infty)$, see \cite[p.~108]{MR2254107}.  Under suitable conditions, the Mehler-Fock transform has an 
inverse. For instance, if $\sqrt t  \hat  f(t)$  is in $L^1(\mathbb R^+)$, then 
$$
f(x) = \mathcal P^{-1}\hat f (x) = \int_0^\infty t \tanh(\pi t) P_{it-1/2}(x) \hat f (t) \, dt, \quad x \geq 1,
$$
see \cite[Thm.~1.9.53]{MR2254107}. 
There is an extensive development of the properties of the  Mehler-Fock transform,  see for instance  \cite{MR2254107}, where further references can be found. It is also well known that the Mehler-Fock transform has applications in Mathematical Physics, where it is used to solve  Dirichlet problems on the sphere and conical surfaces, see  the book by Lebedev \cite{MR0350075}. 

\subsection{An integral representation of the eigenfunctions}

The proof of Theorem 2.1 relies on the following formula, 
\begin{equation}\label{2.6}
	P_{it-1/2}(y)={\cosh  (\pi t) \over \pi }\int_1^\infty {P_{it-1/2}(x) \over x+y} \, dx ,  \quad \text { \rm for } y \geq 1  \text  { \rm and } t \geq 0,
\end{equation}
see the book by Lebedev \cite[p.~202, Ex.~12]{MR0350075}. To prove Theorem~\ref{Hill}, we need an integral representation of the eigenfunctions of the Hilbert matrix. To this end, we consider for each $z \in \mathbb D$ the function 
$$
	\varphi_z(x) =  {1 \over  x(1-z)+1+z } , \quad \text {\rm  for } x \geq 1,
$$
and obtain the following result.

\begin{theorem}\label{Mehler-Fock-rep}
The function $f_{i t +1/2 }$, $t \geq 0  $, can be represented via the Mehler-Fock transform. Indeed,  
$$
f_{it +1/2 }(z)=   {\text {\rm cosh } (\pi t) \over \pi } \mathcal P  \varphi_z(t)  \quad  \text {\rm  for } z \in \D. 
$$ 
In addition, for $0 < \Re \mu \leq 1/2$, we have 
$$
f_\mu(z) = {\sin(\pi \mu)  \over \ \pi} \int_ 1 ^\infty  P _ {\mu-1} (x)   \varphi_ z(x) \, dx .  
$$
\end{theorem}

\begin{proof}
First of all for each $z \in \mathbb D $, we have $\vert \varphi_z(x) \vert = O(x^{-1})$ as $x \to +\infty$ and of course $\varphi_z(x)$ is locally integrable on $[1,\infty)$. Thus $\mathcal P \varphi_z$ is well defined. 

Next assume that $0\leq z<1$. Then $(1+z)/(1-z)\geq 1$. Thus using (2.6)  in the first equality below and (2.6) in the second, we find 
\begin{align*}
	f_{it+1/2}(z) &= {1 \over 1-z} P_{it-1/2}\left ( 1+z \over 1-z \right )\\
	&= {1\over 1-z} {\text { \rm cosh} ( \pi t) \over \pi } \int_1^\infty {P_{it-1/2}(x) \over x+(1+z)/(1-z)} \, dx\\
	&= {\text { \rm cosh} ( \pi t) \over \pi } \int_1^\infty {P_{it-1/2}(x) \over x(1-z)+1+z} \, dx .
\end{align*} 
Since both sides are analytic functions of $z$ on $\mathbb D$, by the identity principle,  it follows that the equality above holds true for each  $z \in \mathbb D$ and we have proved the first part of the statement of the theorem. 

Next, for  $\mu=it+1/2$, {$t\geq 0$},  we  have to show that  the identity 
\begin{equation}\label{2.7}
	f_\mu (z)= {1 \over 1-z } P_{\mu-1}\left ( 1+z \over 1-z \right )
	= {\sin ( \pi \mu) \over \pi} \int_1^\infty {P_{\mu-1}  \left ( x \right ) \over x(1-z)+1+z}  \, dx
\end{equation}
holds true. Observe that due to the assymptotics  of $P_{\mu}$,  see [2], the integral on the right hand side above exists. 

 Since for fixed $z \in \mathbb D$, one can easily show that both sides of (2.7) are entire functions of the parameter  $\mu$, it follows, by the identity principle for analytic functions,  that the identitity in (2.7)  holds  true for each  $\mu \in \mathbb C$.
 The proof is complete.
 \end{proof}

\subsection{Proof of Theorem \ref{Hill}}

We are now in a position to provide a simple proof of  Theorem 2.1.  We need the following integral representation of the Hilbert matrix
\begin{equation}\label{2.8}
	H f(z)= \int_0^1   {f(s) \over 1- s z} \, ds,
\end{equation}
which can be easily verified, see \cite{MR2996436}.   We point out here that in the proof of Theorem~\ref{Mehler-Fock-rep},  we have not made use of the fact that $f_{it+1/2}$ are  eigenfunctions of the Hilbert matrix $H$.    

\begin{proof}[Proof of Theorem \ref{Hill}]
First, consider $\mu =i t+1/2$, $ t \geq 0$. Using Theorem~\ref{Mehler-Fock-rep} in the first and fifth equalities below, Fubini's Theorem in the second equality, formula  \eqref{2.5} in the third equality  and   the change of variables $u=(1+s)/(1-s)$ in the forth equality,  we have 
\begin{align*}
	(H \mathcal P\varphi_z)(t) &= \int_0^1 {1 \over 1-s z} \int_1^\infty{ P_{it-1/2}(x) \over x(1-s)+1+s } \, dx \, ds\\ 
	&= \int_0^1 {1 \over (1-s)(1-s z)} \int_1^\infty{ P_{it-1/2}(x) \over x+(1+s)/(1-s) } \, dx \, ds\\
	&= {\pi \over \cosh (\pi t) }  \int _0^1 {  { P_{it-1/2}\left  ( 1+s \over 1-s \right ) \over(1-s)( 1-sz)}  } ds\\
	&= {\pi \over \cosh (\pi t) } \int_1^\infty{ P_{it-1/2} (u)  \over u (1-z)+1+z} du\\
	&= { \pi \over \cosh(\pi t)} ( \mathcal P \varphi_z)(t).
\end{align*} 
Finally, for general $0<\Re \mu \leq 1/2$, we have to prove 
$$
\int_0^1 {1 \over 1-s z} \int_1^\infty{ P_{\mu}(x) \over x(1-s)+1+s } \, dx \, ds = {\pi \over \sin \pi \mu} \int_1^\infty{ P_{\mu}(x) \over x(1-z)+1+z} \, dx , 
 $$  which follows from the fact that both sides are entire functions of the parameter $\mu$ and  the identity principle for analytic functions, since we have already proved that the equality is true for $\mu=it+1/2$, $t>0$.  
The proof is complete.
\end{proof}

\section{The spectral measure of $H$}

In this section, we will provide a new proof of Magnus' theorem. Indeed,  it is clear that the Hilbert matrix is a self-adjoint operator, so it is enough to find its spectral measure. 

\subsection{The Hardy space of the right half plane and the Gelfand transform}

Before proceeding further we need to introduce the Hardy space of the right  half plane.   We set $\mathbb C^+= \{ z \in \mathbb C: \Re z >0\}$. Then the Hardy space of the right-hand plane $ \mathcal H^2(\mathbb C^+)$ consists of the functions $f$ analytic on $\mathbb C^+$ for which 
$$
	\displaystyle{\sup_{x>0}} \,  {1\over 2 \pi } \int_0^\infty \vert f(x+iy) \vert ^2 dy <\infty.
$$
 It is well known that the functions in $ \mathcal H^2(\mathbb C ^+)$ have boundary values almost everywhere and indeed $\mathcal H^2(\mathbb C^ +)$ is a Hilbert space. In fact,  the norm is given by
$$
	\Vert f \Vert ^2_{\mathcal H^2(\mathbb C^+)}  
	={1\over2  \pi } \int_{-\infty }^{\infty} \vert f(i y)  \vert ^2  \, dy . 
$$
In addition, the map 
$$
	(\mathcal G f)(w)={2\over 1+w} f  \left ( {w-1\over w+1} \right )
$$
is an isometric isomorphism from  $\mathcal H^2(\D)$ onto $\mathcal H^2(\mathbb C^+)$, see \cite[p.~106]{MR0133008}.

Next for each $f \in \mathcal H^2(\D)$, we define  
$$
	\Phi f (t) = \langle f, f _{it+1/2} \rangle _{\mathcal H^2(\D)}= \frac1{2\pi} \int_{-\pi}^{\pi} f(e^{i\theta}) \overline{f _{it+1/2}(e^{i\theta})}\, d\theta.
$$
Since the eigenfunctions $f_{it+1/2}$ are not in $\mathcal H^2(\D)$, in principle $\Phi f $ is not well defined, but in \cite{CMFT} it is shown that $\Phi f$ a well-defined isometric isomorphism between $\mathcal H^2(\D)$ and $L^2(\mathbb R^+,w(t) \, dt )$, where $w(t)=2\pi \tanh \pi t/\sinh \pi t$.  Indeed, $\Phi= M_\phi  \mathcal P G $, where $M_\phi$ is the operator of multiplication by $\phi(t)=(2\pi)^{-1}\cosh \pi t $, see \cite{CMFT}. Now, we have

\begin{theorem}
The Hilbert matrix acting on $\mathcal H^2(\D)$ is isometrically similar to multiplication by $\psi(t)= \pi / \cos \pi t$ acting on $L^2(\mathbb R^+, w(t) \, dt)$. 
\end{theorem}

\begin{proof}
Observing that $f_{it+1/2}(z)$ has real coefficients for $t>0$, it follows that
\begin{align*}
	\Phi H f (t) &= \langle  H f, f_{it+1/2} \rangle _{\mathcal H^2(\D)}\\ 
	&= \langle  f, Hf_{it+1/2} \rangle _{\mathcal H^2(\D) }\\ 
	&= \left  \langle f, {\pi \over \cosh \pi t} f_{it+1/2}  \right \rangle_{\mathcal H^2(\D) }\\
	&= {\pi \over \cosh \pi t } \langle f, f_{it+1/2} \rangle _{\mathcal H^2(\D)}\\
	&= {\pi \over \cosh \pi t} \Phi f (t), 
\end{align*}
which completes the proof.
\end{proof}

Now, as in Magnus~\cite{MR0041358},  it immediately follows that the spectrum of  the Hilbert matrix $H$ is  $[0,\pi]$, since the spectrum of the multiplication operator by $\psi$  is the range of $\varphi$ and of course such an operator has no point spectrum.  

\begin{remark}
Magnus' result is also obtained recently in \cite{MR3479387} but the authors need to show that the Hilbert matrix commutes with a particular Jacobi matrix.
\end{remark}

The following corollary gives the spectral measure of $H$. 

\begin{corollary} 
The  Hilbert matrix $H$ acting on $\mathcal H^2(\D)$ is unitarily similar  to multiplication by $x$ in $L^2([0,\pi], dr)$, where   
$$
	dr(x) ={2\over \pi ^2} { \text {\rm arccosh} \,{ \pi \over x}  } dx.
$$
\end{corollary}
 
\begin{proof}
It is enough to perform the change of variables $x=\pi/\cosh(\pi t)$.
\end{proof}

\section*{Acknowledgments}
This work benefitted from discussions at the workshop ``Riemann-Hilbert problems, Toeplitz matrices, and applications,'' which was held at the American Institute of Mathematics (AIM) in March 2024. The authors thank AIM for providing an excellent scientific environment for their research. This work was further supported in part by Plan Nacional I+D+I  ref. PID2021-127842NB-I00, Junta de Andaluc\'{\i}a FQM-260, and  EPSRC grant EP/X024555/1.

\bibliographystyle{spmpsci}
\bibliography{references}
\end{document}